\documentclass[reqno]{amsart}
\usepackage{amsmath}

\usepackage{graphicx}
\usepackage{amscd}
\usepackage{amsmath, amssymb}
\usepackage{graphics}
\usepackage{graphicx}
\usepackage{epsfig}
\usepackage{xcolor}
\usepackage{soul}
\usepackage{tikz}

\newtheorem{theorem}{Theorem}

\newtheorem{definition}[theorem]{Definition}

\newtheorem{question}[theorem]{Question}

\begin{document}

\title[homogeneous $3$-local representations of the twin groups]{Insights on the homogeneous $3$-local representations of the twin groups} 

\author{Mohamad N. Nasser}

\address{Mohamad N. Nasser\\
         Department of Mathematics and Computer Science\\
         Beirut Arab University\\
         P.O. Box 11-5020, Beirut, Lebanon}
\email{m.nasser@bau.edu.lb}

\begin{abstract}
We provide a complete classification of the homogeneous $3$-local representations of the twin group $T_n$, the virtual twin group $VT_n$, and the welded twin group $WT_n$, for all $n\geq 4$. Beyond this classification, we examine the main characteristics of these representations, particularly their irreducibility and faithfulness. More deeply, we show that all such representations are reducible, and most of them are unfaithful. Also, we find necessary and sufficient conditions of the first two types of the classified representations of $T_n$ to be irreducible in the case $n=4$. The obtained results provide insights into the algebraic structure of these three groups. 
\end{abstract}

\maketitle

\renewcommand{\thefootnote}{}

\vspace*{-0.4cm}

\section{Introduction} 

A Coxeter group $C$ with $r$ generators $c_1,c_2,\ldots, c_r$ is a fundamental group in Algebra that can be presented as follows.
$$C=\langle c_1,c_2, \ldots, c_r \ |\  c_i^2=1,\  (c_ic_j)^{m_{ij}}=1, \ 1\leq i,j \leq r \rangle,$$
where $m_{ij}=1$ if $i=j$ and $m_{ij}\geq 2$ if $i\neq j$. One of the famous Coxeter groups is the twin group on $n$ strands, often denoted by $T_n$, where $n\geq 2$. The twin group was introduced first by G. Shabat and V. Voevodsky \cite{Shabat1990} and it has appeared in history under different names, such as the flat braid group and the planar braid group; see, for instance, \cite{Khovanov1996, Khovanov1997, Merkov1999, Mosto2020, Naik2020, Bellingeri2024}. The twin group $T_n$ is a Coxeter group with $n-1$ generators $s_1,s_2,\ldots, s_{n-1}$ under the following defining relations.
\begin{align*}
&\hspace{0.27cm} s_i^2=1,\hspace{0.75cm} i=1,2,\ldots,n-1,\\
&s_is_j=s_js_i, \hspace{0.3cm} |i-j|\geq 2.
\end{align*}

The twin group $T_n$ has a geometrical interpretation similar to that for the known braid group, namely $B_n$ \cite{Khovanov1996,Khovanov1997}. The twin group and the braid group both arise in the study of permutations of $n$ strands. They are deeply related in concept but also distinct. For instance, the twin generators $s_i$ of $T_n$ represent a transposition-like move between strands $i$ and $i+1$, while the braid generators $\sigma_i$ of $B_n$ represent an over-crossing of strand $i$ over $i+1$. On the other hand, although both groups share the commutative relations $s_is_j=s_js_i$, $1\leq i \leq n-1$, the relations $s_i^2=1$, $1\leq i \leq n-1$, are not satisfied for the generators of $B_n$, while the relations $\sigma_i\sigma_{i+1}\sigma_i=\sigma_{i+1}\sigma_{i}\sigma_{i+1}$, $1\leq i \leq n-2$, are not satisfied for the generators of $T_n$.

\vspace*{0.2cm}

On the far side, studying extensions of braid groups explores deeper algebraic and geometric properties of braids and their applications in other mathematical and physical branches. In \cite{Bardakov2019}, V. Bardakov, M. Singh, and A. Vesnin introduced the virtual twin group on $n$ strands, namely $VT_n$, and the welded twin group on $n$ strands, namely $WT_n$, in analogy with the known virtual and welded braid groups $VB_n$ and $WB_n$. Both groups $VT_n$ and $WT_n$ are group extensions of $T_n$, and they are generated by two families of generators: the twin generators $s_1,s_2,\ldots, s_{n-1}$ and another family of generators denoted by $\rho_1,\rho_2,\ldots, \rho_{n-1}$. We show the relations between the generators in Section 2.

\vspace*{0.2cm}

Group representations and their characteristics allow us to study the group structure, both from algebraic and geometric points of view. In particular, the existence of a faithful representation of a group solves its word problem. A group representation is said to be faithful if it is injective. Till now, there is no discovered faithful representation of $VT_n$ and $WT_n$ unless in very special cases. Another important characteristic of a representation that helps to discover the group structure is its irreducibility. A representation is said to be irreducible if it has no nontrivial subrepresentation. Otherwise it is reducible.

\vspace*{0.2cm}
  
One of the famous and important types of group representations in the field is called the $k$-local representation. A representation of a group $G$ with a finite number of generators $a_1,a_2,\ldots, a_{n-1}$ to $\mathrm{GL}_n(\mathbb{Z}[t^{\pm 1}])$, where $t$ is indeterminate, is said to be $k$-local if the image of the generator $a_i, \ 1\leq i \leq n-1,$ has the form \vspace*{0.1cm}
$$\left( \begin{array}{c|@{}c|c@{}}
   \begin{matrix}
     I_{i-1} 
   \end{matrix} 
      & 0 & 0 \\
      \hline
    0 &\hspace{0.2cm} \begin{matrix}
   		M_i
   		\end{matrix}  & 0  \\
\hline
0 & 0 & I_{n-i-1}
\end{array} \right), \vspace*{0.1cm}$$ 
where $M_i \in \mathrm{GL}_k(\mathbb{Z}[t^{\pm 1}])$ with $k=m-n+2$ and $I_r$ is the $r\times r$ identity matrix. For example, Burau representation \cite{W.B}, Wada representations of types 1 and 2 \cite{wada}, the standard representation \cite{D.T}, and the $F$-representation \cite{V.B2016} are $k$-local representations of $B_n$ with different degrees $k$, while Lawrence-Krammer-Bigelow representation \cite{Law90} is a non-local representation of $B_n$. In history, many $k$-local representations of the braid group and its extensions has been classified and studied (see for instance \cite{Mikha2013,chreif2024,Mayasi2025,M.M.M,Prab}).

\vspace*{0.2cm}

In \cite{Mayasi20251}, T. Mayassi and M. Nasser did a complete classification and examined the main characteristics of the homogeneous $2$-local representations of the twin group $T_n$. The goal of this article is to generalize the work done by Mayassi and Nasser. More precisely, we aim to classify and examine the main characteristics of the homogeneous $3$-local representations of the twin group $T_n$ and its extensions $VT_n$ and $WT_n$ as well.

\vspace*{0.2cm}

The paper is organized in the following way. Section 2 includes the main definitions and previous results we need in our work. In Section 3, we classify all homogeneous $3$-local representations $\tau: T_n\to \mathrm{GL}_{n+1}(\mathbb{C})$ for all $n\geq 4$ (Theorem \ref{3locTn}) and we prove that all such representations are reducible to the degree $n$ (Theorem \ref{irred1}). In addition, we make a complete study for the irreducibility of the two representations $\tau_1: T_4\to \mathrm{GL}_{5}(\mathbb{C})$ and $\tau_2: T_4\to \mathrm{GL}_{5}(\mathbb{C})$ (Theorems \ref{t11}, \ref{t111}, \ref{t21}, \ref{t22}, and \ref{t222}). Likewise, in Section 4, we classify all homogeneous $3$-local representations $\delta: VT_n\to \mathrm{GL}_{n+1}(\mathbb{C})$ and $\gamma: WT_n\to \mathrm{GL}_{n+1}(\mathbb{C})$ for all $n\geq 4$ (Theorems \ref{3locVTn} and \ref{3locWTn}). Also, we prove that all such representations are reducible to the degree $n$ (Theorems \ref{irred12} and \ref{irred122}), and we show that most of such representations are unfaithful (Theorems \ref{fa3locVTn} and \ref{fa3locWTn}).

\section{Preliminaries} 

We start this section by introducing the presentation of the braid group $B_n$ introduced by E. Artin in 1926 \cite{E.A}.

\begin{definition} \label{braid}
\cite{E.A} The braid group $B_n, n\geq 2$, is defined by its generators $\sigma_1,\sigma_2,\ldots,\sigma_{n-1}$ that satisfy the following relations.
\begin{equation} \label{eqs1}
\ \ \ \ \sigma_i\sigma_{i+1}\sigma_i = \sigma_{i+1}\sigma_i\sigma_{i+1}, \hspace{0.45cm} i=1,2,\ldots,n-2,
\end{equation}
\begin{equation} \label{eqs2}
\sigma_i\sigma_j = \sigma_j\sigma_i, \hspace{1.45cm} |i-j|\geq 2.
\end{equation}
\end{definition}

\vspace*{0.1cm}

We then introduce the presentation of the twin group $T_n$ introduced by G. Shabat and V. Voevodsky in 1990 \cite{Shabat1990}. The generators of $T_n$ are called involutions according to the first type of relations given in the following definition.

\vspace*{0.1cm}

\begin{definition}\cite{Shabat1990}
The twin group $T_n, n\geq 2$, is defined by its generators $s_1,s_2,\ldots,s_{n-1}$ that satisfy the following relations.
\begin{equation} \label{eqs1}
\hspace{1.46cm} s_i^2=1,\hspace{0.95cm} i=1,2,\ldots,n-2,
\end{equation}
\begin{equation} \label{eqs2}
s_is_j=s_js_i, \hspace{0.55cm} |i-j|\geq 2.
\end{equation}
\end{definition}

\vspace*{0.1cm}

Note that we have the following particular cases.
\begin{itemize}
\item[•] $T_2=\langle s_1 \ |\ s_1^2=1 \rangle =\mathbb{Z}_2$ is the cyclic group of order $2$.
\item[•] $T_3=\langle s_1,s_2 \ |\ s_1^2=s_2^2=1 \rangle=\mathbb{Z}_2*\mathbb{Z}_2$ is the infinite dihedral group.
\end{itemize}

\vspace*{0.1cm}

We now introduce the presentation of the virtual twin group $VT_n$ introduced by V. Bardakov, M. Singh, and A. Vesnin in 2019 \cite{Bardakov2019}.

\vspace*{0.1cm}

\begin{definition}\cite{Bardakov2019}
The virtual twin group $VT_n, n\geq 2,$ is an extension of $T_n$ that is generated by the generators $s_1,s_2, \ldots,s_{n-1}$ of $T_n$ besides the generators $\rho_1,\rho_2, \ldots, \rho_{n-1}$. In addition to the relations (\ref{eqs1}) and (\ref{eqs2}) of $T_n$, the generators $s_i$ and $\rho_i$, $1\leq i \leq n-1$, of $VT_n$ satisfy the following relations.
\begin{equation} \label{eqs3}
\ \ \ \ \rho_i\rho_{i+1}\rho_i = \rho_{i+1}\rho_i\rho_{i+1}, \hspace{0.6cm} i=1,2,\ldots,n-2,
\end{equation}
\begin{equation} \label{eqs4}
\rho_i\rho_j = \rho_j\rho_i ,\hspace{1.55cm} |i-j|\geq 2,
\end{equation}
\begin{equation} \label{eqs5}
\ \ \ \ \ \ \ \ \ \ \ \ \rho_i^2 = 1 ,\hspace{2cm} i=1,2,\ldots,n-1,
\end{equation}
\begin{equation} \label{eqs6}
s_i\rho_j=\rho_js_i ,\hspace{1.6cm} |i-j|\geq 2,
\end{equation}
\begin{equation} \label{eqs7}
\ \ \ \ \rho_i\rho_{i+1}s_i=s_{i+1}\rho_i\rho_{i+1}, \hspace{0.65cm} i=1,2,\ldots,n-2. \vspace*{0.1cm}
\end{equation}
\end{definition}

\vspace*{0.1cm}

We give now the definition of the welded twin group $WT_n$ introduced also by V. Bardakov, M. Singh, and A. Vesnin in 2019 \cite{Bardakov2019}.

\vspace*{0.1cm}

\begin{definition}\cite{Bardakov2019}
The welded twin group $WT_n, n\geq 2,$ is an extension of $T_n$ that is defined as the quotient of $VT_n$ by adding the following relations.
\begin{equation} \label{eqs8}
\ \ \ \ \rho_is_{i+1}s_i=s_{i+1}s_i\rho_{i+1}, \hspace{0.6cm} i=1,2,\ldots,n-2.
\end{equation}
\end{definition}

\vspace*{0.1cm}

In what follows, we give the concept of $k$-local representations of a group $G$ with a finite number of generators introduced by M. Nasser in 2025 \cite{Nas20251}.

\vspace*{0.1cm}

\begin{definition}
Let $G$ be a group with generators $a_1,a_2,\ldots,a_{n-1}$. A representation $\theta: G \rightarrow \mathrm{GL}_{m}(\mathbb{Z}[t^{\pm 1}])$ is said to be $k$-local if it is of the form
$$\theta(a_i) =\left( \begin{array}{c|@{}c|c@{}}
   \begin{matrix}
     I_{i-1} 
   \end{matrix} 
      & 0 & 0 \\
      \hline
    0 &\hspace{0.2cm} \begin{matrix}
   		M_i
   		\end{matrix}  & 0  \\
\hline
0 & 0 & I_{n-i-1}
\end{array} \right) \hspace*{0.2cm} \text{for} \hspace*{0.2cm} 1\leq i\leq n-1,$$ 
where $M_i \in \mathrm{GL}_k(\mathbb{Z}[t^{\pm 1}])$ with $k=m-n+2$ and $I_r$ is the $r\times r$ identity matrix. The representation $\theta$ is said to be homogeneous if all the matrices $M_i$ are equal.
\end{definition}

\vspace*{0.1cm}

Remark that if $G'$ is a group with $2(n-1)$ generators $a_1,a_2,\ldots,a_{n-1}$ and $b_1,b_2,\ldots,b_{n-1}$, then the concept of $k$-local representations could be extended in the following way.

\vspace*{0.1cm}

\begin{definition}
A $k$-local representation $\theta: G' \rightarrow \mathrm{GL}_{m}(\mathbb{Z}[t^{\pm 1}])$ is a representation of the form
$$\theta(a_i) =\left( \begin{array}{c|@{}c|c@{}}
   \begin{matrix}
     I_{i-1} 
   \end{matrix} 
      & 0 & 0 \\
      \hline
    0 &\hspace{0.2cm} \begin{matrix}
   		M_i
   		\end{matrix}  & 0  \\
\hline
0 & 0 & I_{n-i-1}
\end{array} \right) \text{ and \ } \theta(b_i) =\left( \begin{array}{c|@{}c|c@{}}
   \begin{matrix}
     I_{i-1} 
   \end{matrix} 
      & 0 & 0 \\
      \hline
    0 &\hspace{0.2cm} \begin{matrix}
   		N_i
   		\end{matrix}  & 0  \\
\hline
0 & 0 & I_{n-i-1}
\end{array} \right) $$
for $1\leq i\leq n-1,$ where $M_i,N_i \in \mathrm{GL}_k(\mathbb{Z}[t^{\pm 1}])$ with $k=m-n+2$ and $I_r$ is the $r\times r$ identity matrix. In this case, $\theta$ is homogeneous if all the matrices $M_i$ are equal and all the matrices $N_i$ are equal.
\end{definition}

\vspace*{0.1cm}

The next two definitions are addressed here as examples of famous known $k$-local representations of the braid group $B_n$ of different degrees $k$. The first representation was introduced by W. Burau in 1936 \cite{W.B}, and the second representation was introduced by V. Bardakov and P. Bellingeri in 2016 \cite{V.B2016}.

\vspace*{0.1cm}

\begin{definition} \cite{W.B} \label{defBurau}
The Burau representation $\rho_B: B_n\rightarrow \mathrm{GL}_n(\mathbb{Z}[t^{\pm 1}])$, where $t$ is indeterminate, is the representation defined by
$$\sigma_i\rightarrow \left( \begin{array}{c|@{}c|c@{}}
   \begin{matrix}
     I_{i-1} 
   \end{matrix} 
      & 0 & 0 \\
      \hline
    0 &\hspace{0.2cm} \begin{matrix}
   	1-t & t\\
   	1 & 0\\
\end{matrix}  & 0  \\
\hline
0 & 0 & I_{n-i-1}
\end{array} \right) \hspace*{0.2cm} \text{for} \hspace*{0.2cm} 1\leq i\leq n-1.$$ 
\end{definition}

\vspace*{0.1cm}

\begin{definition} \cite{V.B2016} \label{defF}
The $F$-representation $\rho_F: B_n \rightarrow \mathrm{GL}_{n+1}(\mathbb{Z}[t^{\pm 1}])$, where $t$ is indeterminate, is the representation defined by
$$\sigma_i\rightarrow \left( \begin{array}{c|@{}c|c@{}}
   \begin{matrix}
     I_{i-1} 
   \end{matrix} 
      & 0 & 0 \\
      \hline
    0 &\hspace{0.2cm} \begin{matrix}
   		1 & 1 & 0 \\
   		0 &  -t & 0 \\   		
   		0 &  t & 1 \\
   		\end{matrix}  & 0  \\
\hline
0 & 0 & I_{n-i-1}
\end{array} \right) \hspace*{0.2cm} \text{for} \hspace*{0.2cm} 1\leq i\leq n-1.$$ 
\end{definition}

\vspace*{0.1cm}

From the shapes of the previous two representations, we see that Burau representation is a homogeneous $2$-local, while the $F$-representation is a homogeneous $3$-local. For more information on the characteristics of these two representations, see \cite{E.F} and \cite{M.N2024} respectively.

\vspace*{0.15cm}

On the other side, regarding $k$-local representations of the twin group $T_n$, M. Nasser constructed two $2$-local representations of the twin group $T_n$ and studied their irreducibility and faithfulness in many cases \cite{M.N.twin}. In the following two definitions we introduce these two representations, and we call them $N_1$ and $N_2$ representations.

\vspace*{0.1cm}
 
\begin{definition} \cite{M.N.twin}
The $N_1$-representation $\eta_1: T_n \rightarrow \mathrm{GL}_n(\mathbb{Z}[t^{\pm 1}])$, where $t$ is indeterminate, is the representation defined by
$$s_i\rightarrow \left( \begin{array}{c|@{}c|c@{}}
   \begin{matrix}
     I_{i-1} 
   \end{matrix} 
      & 0 & 0 \\
      \hline
    0 &\hspace{0.2cm} \begin{matrix}
   	1-t & t\\
   	2-t & t-1\\
\end{matrix}  & 0  \\
\hline
0 & 0 & I_{n-i-1}
\end{array} \right) \text{ for } 1\leq i \leq n-1.$$ 
\end{definition}

\vspace*{0.1cm}
 
\begin{definition}\cite{M.N.twin}
The $N_2$-representation $\eta_2: T_n \rightarrow \mathrm{GL}_n(\mathbb{Z}[t^{\pm 1}])$, where $t$ is indeterminate, is the representation defined by
$$s_i \rightarrow \left( \begin{array}{c|@{}c|c@{}}
   \begin{matrix}
     I_{i-1} 
   \end{matrix} 
      & 0 & 0 \\
      \hline
    0 &\hspace{0.2cm} \begin{matrix}
   	0 & f(t)\\
   	f^{-1}(t) & 0\\
\end{matrix}  & 0  \\
\hline
0 & 0 & I_{n-i-1}
\end{array} \right) \text{ for } 1\leq i \leq n-1,$$
where $f(t)$ is invertible in $\mathbb{Z}[t^{\pm 1}]$ and $f^{-1}(t)=\frac{1}{f(t)}$.
\end{definition}

\vspace*{0.15cm}

One of the natural questions that could be addressed regarding $k$-local representations of $T_n$ is given in the following.

\begin{question} \label{qqq}
Let $\tau: T_n \rightarrow \mathrm{GL}_{m}(\mathbb{Z}[t^{\pm 1}])$ be a $k$-local representation of $T_n$. What are the possible forms of $\tau$? And what are their characteristics?
\end{question}

In \cite{Mayasi20251}, T. Mayassi and M. Nasser answered this question in the case $k=2$. In the next section, we answer this question in the case $k=3$. 

\vspace*{0.2cm}

\section{On the $3$-local representations of $T_n$} 

In this section, we classify all homogeneous $3$-local representations of $T_n$ for all $n\geq 4$. Moreover, we prove that they are all reducible to the degree $n$. In addition, we completely study the irreducibility of the first two types in the case $n=4$.

\subsection{Classification of the $3$-local representations of $T_n$}
We start by classifying all homogeneous $3$-local representations of $T_n$ for all $n\geq 4$.

\vspace*{0.1cm}

\begin{theorem}\label{3locTn}
Consider $n\geq 4$ and let $\tau: T_n\to \mathrm{GL}_{n+1}(\mathbb{C})$ be a homogeneous $3$-local representation of $T_n$. Then, $\tau$ is equivalent to one of the following eleven representations $\tau_j$, $1\leq j\leq 11$, where $$\tau_j(s_i)=\left( \begin{array}{c|@{}c|c@{}}
   \begin{matrix}
     I_{i-1} 
   \end{matrix} 
      & 0 & 0 \\
      \hline
    0 &\hspace{0.2cm} \begin{matrix}
   		M_j
   		\end{matrix}  & 0  \\
\hline
0 & 0 & I_{n-i-1}
\end{array} \right)$$
for all $1\leq i \leq n-1$ and the matrices $M_j$'s are given below.
\begin{itemize}
\item[(1)] $M_1=\begin{pmatrix}
     1& 0 & 0\\
     d & -1 & f\\
     0 & 0 & 1
   \end{pmatrix}$, where $d,f \in \mathbb{C}.$
\item[(2)] $M_2=\begin{pmatrix}
     1& 0 & 0\\
     0 & -\sqrt{1-fh} & f\\
     0 & h & \sqrt{1-fh}
   \end{pmatrix}$, where $h,f \in \mathbb{C}.$
\item[(3)] $M_3=\begin{pmatrix}
     1& 0 & 0\\
     0 & \sqrt{1-fh} & f\\
     0 & h & -\sqrt{1-fh}
   \end{pmatrix}$, where $h,f \in \mathbb{C}.$
\item[(4)] $M_4=\begin{pmatrix}
     1& b & 0\\
     0 & -1 & 0\\
     0 & h & 1
   \end{pmatrix}$, where $b,h \in \mathbb{C}.$   
\item[(5)] $M_5=\begin{pmatrix}
     -\sqrt{1-bd}& b & 0\\
     d & \sqrt{1-bd} & 0\\
     0 & 0 & 1
   \end{pmatrix}$, where $b,d \in \mathbb{C}.$ 
\item[(6)] $M_6=\begin{pmatrix}
     \sqrt{1-bd}& b & 0\\
     d & -\sqrt{1-bd} & 0\\
     0 & 0 & 1
   \end{pmatrix}$, where $b,d \in \mathbb{C}.$ 
\item[(7)] $M_7=\begin{pmatrix}
     1& 0 & 0\\
     0 & -1 & 0\\
     0 & 0 & -1
   \end{pmatrix}$.  
\item[(8)] $M_8=\begin{pmatrix}
     -1& 0 & 0\\
     0 & -1 & 0\\
     0 & 0 & 1
   \end{pmatrix}$.  
\item[(9)] $M_9=\begin{pmatrix}
     1& 0 & 0\\
     0 & 1 & 0\\
     0 & 0 & 1
   \end{pmatrix}$. 
\item[(10)] $M_{10}=\begin{pmatrix}
     -1& 0 & 0\\
     0 & -1 & 0\\
     0 & 0 & -1
   \end{pmatrix}$. 
\item[(11)] $M_{11}=\begin{pmatrix}
     -1& 0 & 0\\
     0 & 1 & 0\\
     0 & 0 & -1
   \end{pmatrix}$.    
\end{itemize}
\end{theorem}

\begin{proof}
Set $$\tau(s_i)=\left( \begin{array}{c|@{}c|c@{}}
   \begin{matrix}
     I_{i-1} 
   \end{matrix} 
      & 0 & 0 \\
      \hline
    0 &\hspace{0.2cm} \begin{matrix}
   		M
   		\end{matrix}  & 0  \\
\hline
0 & 0 & I_{n-i-1}
\end{array} \right)$$
for all $1\leq i \leq n-1$, where $$M=\begin{pmatrix}
     a& b & c\\
     d & e & f\\
     g & h & i
   \end{pmatrix}$$ with $a,b,c,d,e,f,g,h,i \in \mathbb{C}.$ Note that $\tau$ preserves the relations of $T_n$, that is
\begin{align*}
&\hspace{0.615cm} \tau(s_i)^2=1,\hspace{1.68cm} i=1,2,\ldots,n-1,\\
&\tau(s_i)\tau(s_j)=\tau(s_j)\tau(s_i), \hspace{0.3cm} |i-j|\geq 2.
\end{align*}
Remark that the only relations we need are: $\tau(s_1)^2=1$ and $\tau(s_1)\tau(s_3)=\tau(s_3)\tau(s_1)$ and all other relations give similar results. Applying these two relations gives a system of twenty four equations with nine unknowns. Two of these equations imply directly that $c=g=0$. The following is the new system of equations after substituting $c$ and $g$ by $0$.
\begin{equation}\label{eq9}
-1+a^2+bd=0,
\end{equation}
\begin{equation}\label{eq10}
ab+be=0,
\end{equation}
\begin{equation}\label{eq11}
bf=0,
\end{equation}
\begin{equation}\label{eq12}
ad+de=0,
\end{equation}
\begin{equation}\label{eq13}
-1+bd+e^2+fh=0,
\end{equation}
\begin{equation}\label{eq14}
ef+fi=0,
\end{equation}
\begin{equation}\label{eq15}
dh=0,
\end{equation}
\begin{equation}\label{eq16}
eh+hi=0,
\end{equation}
\begin{equation}\label{eq17}
-1+fh+i^2=0,
\end{equation}
\begin{equation}\label{eq18}
-f+af=0,
\end{equation}
\begin{equation}\label{eq19}
h-ah=0,
\end{equation}
\begin{equation}\label{eq20}
-b+bi=0,
\end{equation}
\begin{equation}\label{eq21}
d-di=0.
\end{equation}
Equation (\ref{eq11}) implies that $b=0$ or $f=0$ which leads to the following cases.
\begin{itemize}
\item[(a)] \underline{The case $b=0$}. In this case, we have from Equations (\ref{eq9}), (\ref{eq13}), and (\ref{eq17}) that $a^2=1$ and $e^2=i^2=1-fh$. We then consider the following subcases.
\begin{itemize}
\item[(i)] If $h=0$ then we get $a^2=e^2=i^2=1$ and so we have the following.
\begin{itemize}
\item[•] If $a=1, e=1,$ and $i=1$ then we get from Equations (\ref{eq12}) and (\ref{eq14}) that $d=f=0$ and so $\tau$ is equivalent to $\tau_9$.
\item[•] If $a=1, e=-1,$ and $i=1$ then $\tau$ is equivalent to $\tau_1$.
\item[•] If $a=1, e=-1,$ and $i=-1$ then we get from Equations (\ref{eq14}) and (\ref{eq21}) that $d=f=0$ and so $\tau$ is equivalent to $\tau_7$.
\item[•] If $a=1, e=1,$ and $i=-1$ then we get from Equation (\ref{eq21}) that $d=0$ and so $\tau$ is equivalent to a special case of $\tau_3$.
\item[•] If $a=-1, e=1,$ and $i=1$ then we get from Equation (\ref{eq18}) that $f=0$ and so $\tau$ is equivalent to a special case of $\tau_5$.
\item[•] If $a=-1, e=-1,$ and $i=1$ then we get from Equations (\ref{eq12}) and (\ref{eq18}) that $d=f=0$ and so $\tau$ is equivalent to $\tau_8$.
\item[•] If $a=-1, e=-1,$ and $i=-1$ then we get from Equations (\ref{eq18}) and (\ref{eq21}) that $d=f=0$ and so $\tau$ is equivalent to $\tau_{10}$.
\item[•] If $a=-1, e=1,$ and $i=-1$ then we get from Equations (\ref{eq18}) and (\ref{eq21}) that $d=f=0$ and so $\tau$ is equivalent to $\tau_{11}$.
\end{itemize}
\item[(ii)] If $h\neq 0$ then we get from Equations (\ref{eq15}), (\ref{eq16}), and (\ref{eq19}) that $d=0$, $e=-i$, and $a=1$ and so $\tau$ is equivalent to special cases of $\tau_2$ or $\tau_3$.
\end{itemize}
\item[(b)] \underline{The case $f=0$}. In this case, we have from Equations (\ref{eq9}), (\ref{eq13}), and (\ref{eq17}) that $a^2=e^2=1-bd$ and $i^2=1$. We then consider the following subcases.
\begin{itemize}
\item[(i)] If $d=0$ then we get $a^2=e^2=i^2=1$ and so we have the following.
\begin{itemize}
\item[•] If $a=1, e=1,$ and $i=1$ then we get from Equations (\ref{eq10}) and (\ref{eq16}) that $b=h=0$ and so $\tau$ is equivalent to $\tau_9$.
\item[•] If $a=1, e=-1,$ and $i=1$ then $\tau$ is equivalent to $\tau_4$.
\item[•] If $a=1, e=-1,$ and $i=-1$ then we get from Equations (\ref{eq16}) and (\ref{eq20}) that $b=h=0$ and so $\tau$ is equivalent to $\tau_7$.
\item[•] If $a=1, e=1,$ and $i=-1$ then we get from Equation (\ref{eq10}) that $b=0$ and so $\tau$ is equivalent to a special case of $\tau_3$.
\item[•] If $a=-1, e=1,$ and $i=1$ then we get from Equation (\ref{eq16}) that $h=0$ and so $\tau$ is equivalent to a special case of $\tau_5$.
\item[•] If $a=-1, e=-1,$ and $i=1$ then we get from Equations (\ref{eq10}) and (\ref{eq19}) that $b=h=0$ and so $\tau$ is equivalent to $\tau_8$.
\item[•] If $a=-1, e=-1,$ and $i=-1$ then we get from Equations (\ref{eq10}) and (\ref{eq16}) that $b=h=0$ and so $\tau$ is equivalent to $\tau_{10}$.
\item[•] If $a=-1, e=1,$ and $i=-1$ then we get from Equations (\ref{eq19}) and (\ref{eq20}) that $b=h=0$ and so $\tau$ is equivalent to $\tau_{11}$.
\end{itemize}
\item[(ii)] If $d \neq 0$ then we get from Equations (\ref{eq12}), (\ref{eq15}), and (\ref{eq21}) that $h=0$, $a=-e$, and $i=1$ and so $\tau$ is equivalent to special cases of $\tau_5$ or $\tau_6$.
\end{itemize}
\end{itemize}
\end{proof}

\subsection{On the irreducibility of the $3$-local representations of $T_n$}
In this subsection, we prove that every homogeneous $3$-local representation of $T_n$ is reducible for all $n\geq 4$.
 
\begin{theorem}\label{irred1}
Consider $n\geq 4$ and let $\tau: T_n\to \mathrm{GL}_{n+1}(\mathbb{C})$ be a homogeneous $3$-local representation of $T_n$. Then, $\tau$ is reducible.
\end{theorem}
\begin{proof}
According to Theorem \ref{3locTn}, we know that $\tau$ is equivalent to one of the representations $\tau_j, 1\leq j \leq 11$, and so we consider the following cases.
\begin{itemize}
\item[(1)] In the case $\tau$ is equivalent to $\tau_1$ we have two subcases.
\begin{itemize}
\item[•] If $f\neq 0$ then we see that the vector $(1,x,x^2,\ldots,x^n)^T$, where $x=\frac{1-\sqrt{1-df}}{f}$ and $T$ is the transpose, is invariant under $\tau_1(s_i)$ for all $1\leq i \leq n-1$. Thus, $\tau_1$ is reducible and so $\tau$ is reducible.
\item[•] If $f=0$ then we see that the vector $(0,\ldots,0,1)^T$ is invariant under $\tau_1(s_i)$ for all $1\leq i \leq n-1$. Thus, $\tau_1$ is reducible and so $\tau$ is reducible.
\end{itemize}
\item[(2)] In the case $\tau$ is equivalent to $\tau_j, 2\leq j \leq 4,$ we see that the vector $(1,0,\ldots,0)^T$ is invariant under $\tau_j(s_i)$ for all $1\leq i \leq n-1$. Thus, $\tau_j$ is reducible and so $\tau$ is reducible.
\item[(3)] In the case $\tau$ is equivalent to $\tau_j, 5\leq j \leq 6,$ we see that the vector $(0,\ldots,0,1)^T$ is invariant under $\tau_j(s_i)$ for all $1\leq i \leq n-1$. Thus, $\tau_j$ is reducible and so $\tau$ is reducible.
\item[(4)] If $\tau$ is equivalent to $\tau_j, 7\leq j \leq 11,$ then clearly $\tau$ is reducible.
\end{itemize}
\end{proof}

Notice that, from the shapes of the representations $\tau_j, 1\leq j \leq 11,$ we can see that many representations share similar shapes. So, in the next two subsections, we completely study the irreducibility of the first two representations $\tau_1$ and $\tau_2$ in the case $n=4$, and the work would be similar for the remaining representations.

\subsection{The irreducibility of $\tau_1$ in the case $n=4$}

In this subsection, we answer the question of the irreducibility of the representation $\tau_1$ given in Theorem \ref{3locTn} in the case $n=4$. We take here $f\neq 0$ in $\tau_1$ since the case $f=0$ is straightforward. We start with the following theorem.

\begin{theorem} \label{t11}
Consider the representation $\tau_1: T_4\to \mathrm{GL}_{5}(\mathbb{C})$ given in Theorem \ref{3locTn} with $f\neq 0$. Then, $\tau_1$ has a composition factor, namely $\tau_1^{(1)}:T_4 \to \mathrm{GL}_3(\mathbb{C})$, which is given by acting on the generators of $T_4$, $s_i, 1\leq i \leq 3,$ as follows. 
$$\tau_1^{(1)}(s_1)=\begin{pmatrix}
     -1& f & 0\\
     0 & 1 & 0\\
     0 & 0 & 1
   \end{pmatrix},$$
$$\tau_1^{(1)}(s_2)=\begin{pmatrix}
     1& 0 & 0\\
     d & -1 & f\\
     0 & 0 & 1
   \end{pmatrix}, $$
and
$$\tau_1^{(1)}(s_3)=\begin{pmatrix}
     1& 0 & 0\\
     0 & 1 & 0\\
     0 & d & -1
   \end{pmatrix}.$$
\end{theorem}

\begin{proof}
By Theorems \ref{3locTn} and \ref{irred1}, we have seen that the vector $X=(1,x,x^2,x^3,x^4)^T$, where $x=\frac{1-\sqrt{1-df}}{f}$, is invariant under the matrices $\tau_1(s_i)$ for all $1\leq i \leq 3$. Consider the new basis $\{X,e_2,e_3,e_4,e_5\}$ of $\mathbb{C}^5$ where $e_i's$ are the standard unit vectors of $\mathbb{C}^5$. We can see that\vspace*{0.05cm}\\
$\tau_1(s_1)(X)=X,$\\
$\tau_1(s_1)(e_2)=-e_2$,\\
$\tau_1(s_1)(e_3)=(f)e_2+e_3$,\\
$\tau_1(s_1)(e_4)=e_4$,\\
$\tau_1(s_1)(e_5)=e_5$,\\
$\tau_1(s_2)(X)=X,$\\
$\tau_1(s_2)(e_2)=e_2+(d)e_3$,\\
$\tau_1(s_2)(e_3)=-e_3$,\\
$\tau_1(s_2)(e_4)=(f)e_3+e_4$,\\
$\tau_1(s_2)(e_5)=e_5$,\\
$\tau_1(s_3)(X)=X,$\\
$\tau_1(s_3)(e_2)=e_2$,\\
$\tau_1(s_3)(e_3)=e_3+(d)e_4$,\\
$\tau_1(s_3)(e_4)=-e_4$,\\
$\tau_1(s_3)(e_5)=(f)e_4+e_5$.\vspace*{0.05cm} \\
We write the representation $\tau_1$ on the new basis $\{X,e_2,e_3,e_4,e_5\}$ and we get the following. 
$$\tau_1(s_1)=\begin{pmatrix}
     1& 0 & 0 & 0 & 0\\
     0& -1 & f & 0 & 0\\
     0& 0 & 1 & 0 & 0\\
     0& 0 & 0 & 1 & 0\\
     0& 0 & 0 & 0 & 1\\
   \end{pmatrix},$$
$$\tau_1(s_2)=\begin{pmatrix}
     1& 0 & 0 & 0 & 0\\
     0& 1 & 0 & 0 & 0\\
     0& d & -1 & f & 0\\
     0& 0 & 0 & 1 & 0\\
     0& 0 & 0 & 0 & 1\\
   \end{pmatrix}, $$
and
$$\tau_1(s_3)=\begin{pmatrix}
     1& 0 & 0 & 0 & 0\\
     0& 1 & 0 & 0 & 0\\
     0& 0 & 1 & 0 & 0\\
     0& 0 & d & -1 & f\\
     0& 0 & 0 & 0 & 1\\
   \end{pmatrix}.$$
By removing the first row and first column, as well as the last row and last column, in each of the matrices above, we obtain the desired result. This reduction is justified by the fact that the subspace spanned by \(\{X, e_2, e_3, e_4\}\) is invariant under the group action; indeed, in the new basis, the generators have fifth row \((0,0,0,0,1)\), which allows us to restrict to a $4$-dimensional invariant subspace prior to quotienting out the invariant vector \(X\).
\end{proof}

Now we study the irreducibility of the representation $\tau_1^{(1)}:T_4 \to \mathrm{GL}_3(\mathbb{C})$.

\begin{theorem} \label{t111}
Consider the representation $\tau_1^{(1)}:T_4 \to \mathrm{GL}_3(\mathbb{C})$ given in Theorem \ref{t11}. We have the following two cases:
\begin{itemize}
\item[(1)] If $d=0$, then $\tau_1^{(1)}$ is reducible.
\item[(2)] If $d\neq 0$, then $\tau_1^{(1)}$ is irreducible if and only if $d\neq \frac{2}{f}$.
\end{itemize}
\end{theorem}

\begin{proof}
Recall that the representation $\tau_1^{(1)}:T_4 \to \mathrm{GL}_3(\mathbb{C})$ is given by acting on the generators of $T_4$, $s_i, 1\leq i \leq 3,$ as follows. 
$$\tau_1^{(1)}(s_1)=\begin{pmatrix}
     -1& f & 0\\
     0 & 1 & 0\\
     0 & 0 & 1
   \end{pmatrix},$$
$$\tau_1^{(1)}(s_2)=\begin{pmatrix}
     1& 0 & 0\\
     d & -1 & f\\
     0 & 0 & 1
   \end{pmatrix}, $$
and
$$\tau_1^{(1)}(s_3)=\begin{pmatrix}
     1& 0 & 0\\
     0 & 1 & 0\\
     0 & d & -1
   \end{pmatrix}.$$
We consider each case separately in the following.
\begin{itemize}
\item[(1)] If $d=0$ then we clearly see that $(1,0,0)^T$ is invariant under $\tau_1^{(1)}(s_i)$ for all $1\leq i \leq 3$ and so $\tau_1^{(1)}$ is reducible. 
\item[(2)] Suppose that $d\neq 0$. For the necessary condition, if $d=\frac{2}{f}$, then direct computations implies that the vector $(\frac{f}{2},1,\frac{1}{f})^T$ is invariant under $\tau_1^{(1)}(s_i)$ for all $1\leq i \leq 3$ and so we get that $\tau_1^{(1)}$ is reducible. Now, for the sufficient condition, suppose that $d\neq \frac{2}{f}$, and assume to get a contradiction that $\tau_1^{(1)}$ is reducible. Let $U$ be a nontrivial invariant subspace of $\mathbb{C}^3$ and let $u=(u_1,u_2,u_3)^T \in U$ be a nonzero element. Then, we have the following.
$$v_1=\tau_1^{(1)}(s_1)(u)-u=(-2u_1+fu_2)e_1 \in U,$$
$$v_2=\tau_1^{(1)}(s_2)(v_1)-v_1=d(-2u_1+fu_2)e_2 \in U,$$
$$v_3=\tau_1^{(1)}(s_3)(v_2)-v_2=d^2(-2u_1+fu_2)e_3 \in U.$$
As $d\neq 0$, we obtain that $-2u_1+fu_2=0$ since otherwise we get that $e_1,e_2,e_3 \in U$, which is a contradiction as $U$ is nontrivial. So, $u_1=\frac{f}{2}u_2$ and so $u=(\frac{f}{2}u_2,u_2,u_3)^T$. Similarly, we also have the following.
$$w_1=\tau_1^{(1)}(s_3)(u)-u=(du_2-2u_3)e_3 \in U,$$
$$w_2=\tau_1^{(1)}(s_2)(w_1)-w_1=f(du_2-2u_3)e_2 \in U,$$
$$w_3=\tau_1^{(1)}(s_1)(w_2)-w_2=f^2(du_2-2u_3)e_1 \in U.$$
Similarly, as $f\neq 0$ and $U$ is nontrivial, we get that $du_2-2u_3=0$. So, $u_3=\frac{d}{2}u_2$, which gives that $u=(\frac{f}{2}u_2,u_2,\frac{d}{2}u_2)^T$. Therefore, as $u$ is a nonzero element in $U$, we conclude that $$U=\langle z=\left(\frac{f}{2},1,\frac{d}{2} \right)^T \rangle.$$ Now, we have that $\tau_1^{(1)}(s_1)(z)-z=(-2+df)e_2 \in U$ with $-2+df\neq 0$, which implies that $e_2 \in U$, and so we have $e_2$ is a multiples of $z$, a clear contradiction as $f\neq 0$ and $d\neq 0$. Hence,  $\tau_1^{(1)}$ is irreducible in this case, as required.
\end{itemize} 
\end{proof}

\vspace*{0.1cm}

\subsection{The irreducibility of $\tau_2$ in the case $n=4$}

In this subsection, we answer the question of the irreducibiliy of the representation $\tau_2$ given in Theorem \ref{3locTn} in the case $n=4$. We start with the following theorem.

\begin{theorem} \label{t21}
Consider the representation $\tau_2: T_4\to \mathrm{GL}_{5}(\mathbb{C})$ given in Theorem \ref{3locTn}. Then, $\tau_2$ has a composition factor, namely $\tau_2^{(1)}:T_4 \to \mathrm{GL}_4(\mathbb{C})$, which is given by acting on the generators of $T_4$, $s_i, 1\leq i \leq 3,$ as follows. 
$$\tau_2^{(1)}(s_i)=\left( \begin{array}{c|@{}c|c@{}}
   \begin{matrix}
     I_{i-1} 
   \end{matrix} 
      & 0 & 0 \\
      \hline
    0 &\hspace{0.2cm} \begin{matrix}
   		-\sqrt{1-fh} & f\\
   		h & \sqrt{1-fh}\\
   		\end{matrix}  & 0  \\
\hline
0 & 0 & I_{3-i}
\end{array} \right).$$
\end{theorem}
\begin{proof}
By Theorems \ref{3locTn} and \ref{irred1}, we have seen that the vector $(1,0,\ldots, 0)^T$ is invariant under the matrices $\tau_2(s_i)$ for all $1\leq i \leq 3$. So, eliminating the first row and the first column of each of the matrices $\tau_2(s_i)$ gives the required result.
\end{proof}

\begin{theorem} \label{t22}
The representation $\tau_2^{(1)}:T_4 \to \mathrm{GL}_4(\mathbb{C})$ given in Theorem \ref{t21} is reducible. Moreover, we have the following.
\begin{itemize}
\item[•] If $f=0$, then the composition factor of $\tau_2^{(1)}$, namely $\tau_2^{(2)}:T_4 \to \mathrm{GL}_3(\mathbb{C})$, is the mapping that takes every generator to a lower triangular matrix.
\item[•] If $f\neq 0$, then the composition factor of $\tau_2^{(1)}$ namely, $\tau_2^{(2)}:T_4 \to \mathrm{GL}_3(\mathbb{C})$, is given by acting on the generators $s_i, 1\leq i \leq 3,$ as follows.
$$\tau_2^{(2)}(s_1)=\begin{pmatrix}
     -1& 0 & 0\\
     -\frac{\left(\sqrt{1-f h}+1\right)^2}{f} & 1 & 0\\
     -\frac{\left(\sqrt{1-f h}+1\right)^3}{f^2} & 0 & 1
   \end{pmatrix},$$
$$\tau_2^{(2)}(s_2)=\begin{pmatrix}
     -\sqrt{1-fh}& f & 0\\
     h & \sqrt{1-fh} & 0\\
     0 & 0 & 1
   \end{pmatrix}, $$
and
$$\tau_2^{(2)}(s_3)=\begin{pmatrix}
     1& 0 & 0\\
     0 & -\sqrt{1-fh} & f\\
     0 & h & \sqrt{1-fh}
   \end{pmatrix}.$$
\end{itemize}
\end{theorem}

\begin{proof}
We consider two cases in the following.
\begin{itemize}
\item[•] If $f=0$, then clearly we can see that the vector $(1,0,0,0)^T$ is invariant under $\tau_2^{(1)}(s_i)$ for all $1\leq i \leq 3$. Thus, $\tau_2^{(1)}$ is reducible. Now, eliminating the first row and the first column implies that the composition factor of $\tau_2^{(1)}$, namely $\tau_2^{(2)}:T_4 \to \mathrm{GL}_3(\mathbb{C})$, is the mapping that takes every generator to a lower triangular matrix.
\item[•] Suppose that $f\neq 0$. We can see that the vector $X=(1,x,x^2,x^3)^T$, where $x=\frac{1+\sqrt{1-fh}}{f}$, is invariant under $\tau_2^{(1)}(s_i)$ for all $1\leq i \leq 3$. Thus, $\tau_2^{(1)}$ is reducible. Now, consider the new basis $\{X,e_2,e_3,e_4\}$ of $\mathbb{C}^4$ where $e_i's$ are the standard unit vectors of $\mathbb{C}^4$. We can see that\vspace*{0.05cm}\\
$\tau_2^{(1)}(s_1)(X)=X,$\\
$\tau_2^{(1)}(s_1)(e_2)=(f)e_1+(\sqrt{1-fh})e_2\\ \hspace*{1.75cm} =f(X-(x)e_2-(x^2)e_3-(x^3)e_4)+(\sqrt{1-fh})e_2 \\ \hspace*{1.75cm} =(f)X+(\sqrt{1-fh}-fx)e_2-(x^2f)e_3-(x^3f)e_4\\ \hspace*{1.75cm} =(f)X-e_2-\left(\frac{\left(\sqrt{1-f h}+1\right)^2}{f}\right) e_3-\left(\frac{\left(\sqrt{1-f h}+1\right)^3}{f^2}\right)e_4,$\\
$\tau_2^{(1)}(s_1)(e_3)=e_3$,\\
$\tau_2^{(1)}(s_1)(e_4)=e_4$,\\
$\tau_2^{(1)}(s_2)(X)=X$,\\
$\tau_2^{(1)}(s_2)(e_2)=(-\sqrt{1-fh})e_2+(h)e_3$,\\
$\tau_2^{(1)}(s_2)(e_3)=(f)e_2+(\sqrt{1-fh})e_3$,\\
$\tau_2^{(1)}(s_2)(e_4)=e_4$,\\
$\tau_2^{(1)}(s_3)(X)=X$,\\
$\tau_2^{(1)}(s_3)(e_2)=e_2$,\\
$\tau_2^{(1)}(s_3)(e_3)=-(\sqrt{1-fh})e_3+(h)e_4$,\\
$\tau_2^{(1)}(s_3)(e_4)=(f)e_3+(\sqrt{1-fh})e_4$.\vspace*{0.05cm}
\end{itemize}
Writing $\tau_2^{(1)}$ in the new basis $\{X,e_2,e_3,e_4\}$ of $\mathbb{C}^4$ and eliminating the first row and the first column gives our required result for $\tau_2^{(2)}$.
\end{proof}

Now we find necessary and sufficient conditions for the representation $\tau_2^{(2)}:T_4 \to \mathrm{GL}_3(\mathbb{C})$, in the case $f\neq 0$, to be irreducible. The case $f=0$ is straightforward.

\begin{theorem} \label{t222}
Consider the representation $\tau_2^{(2)}:T_4 \to \mathrm{GL}_3(\mathbb{C})$ given in Theorem \ref{t22} with $f\neq 0$. Then, $\tau_2^{(2)}$ is irreducible if and only if $h\neq 0$.
\end{theorem}

\begin{proof}
For the necessary condition, if $h=0$ then we can see that the vector $(1,\frac{2}{f},\frac{4}{f^2})^T$ is invariant under $\tau_2^{(2)}(s_i)$ for all $1 \leq i \leq 3$, and so $\tau_2^{(2)}$ is reducible. Now, for the sufficient condition, we have $h\neq 0$. Assume to  get a contradiction that $\tau_2^{(2)}$ is reducible and let $U$ be a nontrivial subspace of $\mathbb{C}^3$ which is invariant under $\tau_2^{(2)}$. Then, the dimension of $U$ is $1$ or $2$. We consider each case separately.
\begin{itemize}
\item[(1)] The case $\dim(U)=1$. In this case, set $U=\langle u \rangle$ where $u=\alpha_1e_1+\alpha_2e_2+\alpha_3e_3$, where $e_i's$ are the standard unit vectors of $\mathbb{C}^3$. We have the following
$$\tau_2^{(2)}(s_2)(u)=\left(-(\sqrt{1-fh})\alpha_1+f\alpha_2, h\alpha_1+(\sqrt{1-fh})\alpha_2, \alpha_3 \right)^T\in U,$$
which implies that there exists a scalar $\beta$ such that 
$$\left(-(\sqrt{1-fh})\alpha_1+f\alpha_2, h\alpha_1+(\sqrt{1-fh})\alpha_2, \alpha_3 \right)^T=\beta u=\beta(\alpha_1, \alpha_2, \alpha_3).$$
This means that $\beta =1$ and so we get the following two equations.
\begin{equation}
-(\sqrt{1-fh})\alpha_1+f\alpha_2=\alpha_1,
\end{equation}
\begin{equation}
h\alpha_1+(\sqrt{1-fh})\alpha_2=\alpha_2.
\end{equation}
Doing the same work for $\tau_2^{(2)}(s_3)$ instead of $\tau_2^{(2)}(s_2)$ gives the following two additional equations.
\begin{equation}
-(\sqrt{1-fh})\alpha_2+f\alpha_3=\alpha_2,
\end{equation}
\begin{equation}
h\alpha_2+(\sqrt{1-fh})\alpha_3=\alpha_3.
\end{equation}
This gives that
$$\alpha_2=\dfrac{1+\sqrt{1-fh}}{f} \alpha_1$$
and
$$\alpha_3=\left( \dfrac{1+\sqrt{1-fh}}{f}\right)^2\alpha_1,$$
and so we have that 
$$u=\left( 1,\dfrac{1+\sqrt{1-fh}}{f},\left( \dfrac{1+\sqrt{1-fh}}{f}\right)^2 \right).$$
Now, we have 
$$\tau_2^{(2)}(s_1)(u)=-\left(1, \dfrac{1-fh+\sqrt{1-fh}}{f}, (\sqrt{1-fh})\dfrac{1+\sqrt{1-fh}}{f^2} \right)^T\in U,$$
which implies that there exists a scalar $\xi$ such that 
$$-\left(1, \dfrac{1-fh+\sqrt{1-fh}}{f}, (\sqrt{1-fh})\dfrac{1+\sqrt{1-fh}}{f^2} \right)^T=\xi u.$$
This implies that $\xi=-1$ and so we get that
$$\dfrac{1+\sqrt{1-fh}}{f}=\dfrac{1-fh+\sqrt{1-fh}}{f},$$ 
which gives that $fh=0$, a contradiction since both $f$ and $h$ are nonzero. \vspace*{0.2cm}
\item[(2)] The case $\dim(U)=2$. In this case, set $U=\langle u,v \rangle$ where $u=u_1e_1+u_2e_2+u_3e_3$ and $v=v_1e_1+v_2e_2+v_3e_3$. We consider here two subcases.
\begin{itemize}
\item[•] If $u_1=v_1=0$, then we have $\tau_2^{(2)}(s_2)(u)=(fu_2)e_1+(\sqrt{1-fh})(u_3)e_2+u_3\in U$. But $f\neq 0$, which implies that $u_2=0$. In a similar way we can see that that $v_2=0$. So we get that both vectors $u$ and $v$ are multiple of $e_3$, and so they are linearly dependent, a contradiction.
\item[•] Suppose, without loss of generality, that $u_1\neq 0$. Then, we have the following.
\\
\\
\\
\\ 
$$W_1=\frac{1}{u_1}\left(\tau_2^{(2)}(s_1)(u)-u\right)=- \left(2, \dfrac{(\sqrt{1-fh}+1)^2}{f},\dfrac{(\sqrt{1-fh}+1)^3}{f^2}  \right)^T \in U$$
and so
$$W_2=\frac{1}{fh}\left(\tau_2^{(2)}(s_2)(W_1)-W_1\right)=\left( 1, \dfrac{\sqrt{1-fh}-1}{f},0  \right)^T \in U$$
and then 
$$W_3=\frac{1}{h}\left(\tau_2^{(2)}(s_3)(W_2)-W_2\right)=\left( 0, 1, \dfrac{\sqrt{1-fh}-1}{f}  \right)^T \in U.$$
Now, elementary linear algebra gives that $W_1, W_2,$ and $W_3$ are three linearly independent vectors in $U$ which is of dimension $2$, a contradiction.
\end{itemize}
\end{itemize} 
Therefore, $\tau_2^{(2)}$ is irreducbile in this case and the proof is completed.
\end{proof}

\vspace*{0.15cm}

\section{On the $3$-local representations of $VT_n$ and $WT_n$} 

In this section, we classify all homogeneous $3$-local representations of $VT_n$ and $WT_n$ for all $n\geq 4$. Moreover, we study the irreducibility and faithfulness of these representations in most cases.

\vspace*{0.1cm}

\subsection{Classification of the $3$ local representations of $VT_n$ and $WT_n$}

We start by classifying all homogeneous $3$-local representations of $VT_n$ for all $n\geq 4$.

\vspace*{0.1cm}

\begin{theorem}\label{3locVTn}
Consider $n\geq 4$ and let $\delta: VT_n\to \mathrm{GL}_{n+1}(\mathbb{C})$ be a homogeneous $3$-local representation of $VT_n$. Then, $\delta$ is equivalent to one of the following fourteen representations $\delta_j$, $1\leq j\leq 14$, where $$\delta_j(s_i)=\left( \begin{array}{c|@{}c|c@{}}
   \begin{matrix}
     I_{i-1} 
   \end{matrix} 
      & 0 & 0 \\
      \hline
    0 &\hspace{0.2cm} \begin{matrix}
   		M_j
   		\end{matrix}  & 0  \\
\hline
0 & 0 & I_{n-i-1}
\end{array} \right)
\text{ and  }\ \delta_j(\rho_i)=\left( \begin{array}{c|@{}c|c@{}}
   \begin{matrix}
     I_{i-1} 
   \end{matrix} 
      & 0 & 0 \\
      \hline
    0 &\hspace{0.2cm} \begin{matrix}
   		N_j
   		\end{matrix}  & 0  \\
\hline
0 & 0 & I_{n-i-1}
\end{array} \right)$$
for all $1\leq i \leq n-1$,  and the matrices $M_j$'s and $N_j$'s are given below.
\begin{itemize}
\item[(1)] $M_1=\begin{pmatrix}
     1& 0 & 0\\
     0 & e & \frac{1-e^2}{h}\\
     0 & h & -e
   \end{pmatrix}$ and $N_1=\begin{pmatrix}
     1& 0 & 0\\
     0 & 0 & p\\
     0 & \frac{1}{p} & 0
   \end{pmatrix}$, where $e \in \mathbb{C}$, $h,p\in \mathbb{C}^*$.
\item[(2)] $M_2=\begin{pmatrix}
     -e& \frac{1-e^2}{d} & 0\\
     d & e & 0\\
     0 & 0 & 1
   \end{pmatrix}$ and $N_2=\begin{pmatrix}
     0& k & 0\\
     \frac{1}{k} & 0 & 0\\
     0 & 0 & 1
   \end{pmatrix}$, where $e \in \mathbb{C}$, $d,k\in \mathbb{C}^*$.
\item[(3)] $M_3=\begin{pmatrix}
     1& 0 & 0\\
     d & -1 & 2p-dp^2\\
     0 & 0 & 1
   \end{pmatrix}$ and $N_3=\begin{pmatrix}
     1& 0 & 0\\
     \frac{1}{p} & -1 & p\\
     0 & 0 & 1
   \end{pmatrix}$, where $d \in \mathbb{C}$, $p\in \mathbb{C}^*$.
\item[(4)] $M_4=\begin{pmatrix}
     1& 2k-hk^2 & 0\\
     0 & -1 & 0\\
     0 & h & 1
   \end{pmatrix}$ and $N_4=\begin{pmatrix}
     1& k & 0\\
     0 & -1 & 0\\
     0 & \frac{1}{k} & 1
   \end{pmatrix}$, where $h \in \mathbb{C}$, $k\in \mathbb{C}^*$.
\item[(5)] $M_5=\begin{pmatrix}
     1& b & 0\\
     0 & -1 & 0\\
     0 & 0 & 1
   \end{pmatrix}$ and $N_5=\begin{pmatrix}
     0 & k & 0\\
     \frac{1}{k} & 0 & 0\\
     0 & 0 & 1
   \end{pmatrix}$, where $b \in \mathbb{C}$, $k\in \mathbb{C}^*$.
\item[(6)] $M_6=\begin{pmatrix}
     -1& b & 0\\
     0 & 1 & 0\\
     0 & 0 & 1
   \end{pmatrix}$ and $N_6=\begin{pmatrix}
     0 & k & 0\\
     \frac{1}{k} & 0 & 0\\
     0 & 0 & 1
   \end{pmatrix}$, where $b \in \mathbb{C}$, $k\in \mathbb{C}^*$.
\item[(7)] $M_7=\begin{pmatrix}
     1& 0 & 0\\
     0 & 1 & 0\\
     0 & 0 & 1
   \end{pmatrix}$ and $N_7=\begin{pmatrix}
     1 & 0 & 0\\
     \frac{1}{p} & 0 & p\\
     0 & 0 & 1
   \end{pmatrix}$, where $p\in \mathbb{C}^*$.   
\item[(8)] $M_8=\begin{pmatrix}
     1& 0 & 0\\
     0 & 1 & 0\\
     0 & 0 & 1
   \end{pmatrix}$ and $N_8=\begin{pmatrix}
     1 & k & 0\\
     0 & -1 & 0\\
     0 & \frac{1}{k} & 1
   \end{pmatrix}$, where $k\in \mathbb{C}^*$.
\item[(9)] $M_9=\begin{pmatrix}
     1& 0 & 0\\
     0 & -1 & 0\\
     0 & 0 & -1
   \end{pmatrix}$ and $N_9=\begin{pmatrix}
     1 & 0 & 0\\
     0 & 0 & p\\
     0 & \frac{1}{p} & 1
   \end{pmatrix}$, where $p \in \mathbb{C}^*$. 
\item[(10)] $M_{10}=\begin{pmatrix}
     1& 0 & 0\\
     0 & -1 & 0\\
     0 & 0 & 1
   \end{pmatrix}$ and $N_{10}=\begin{pmatrix}
     1 & 0 & 0\\
     0 & 0 & p\\
     0 & \frac{1}{p} & 1
   \end{pmatrix}$, where $p \in \mathbb{C}^*$.   
\item[(11)] $M_{11}=\begin{pmatrix}
     1& 0 & 0\\
     0 & 1 & 0\\
     0 & 0 & 1
   \end{pmatrix}$ and $N_{11}=\begin{pmatrix}
     1 & 0 & 0\\
     0 & 0 & p\\
     0 & \frac{1}{p} & 1
   \end{pmatrix}$, where $p \in \mathbb{C}^*$.     
\item[(12)] $M_{12}=\begin{pmatrix}
     -1& 0 & 0\\
     0 & -1 & 0\\
     0 & 0 & 1
   \end{pmatrix}$ and $N_{12}=\begin{pmatrix}
     1 & k & 0\\
     \frac{1}{k} & 0 & 0\\
     0 & 0 & 1
   \end{pmatrix}$, where $k \in \mathbb{C}^*$.
\item[(13)] $M_{13}=\begin{pmatrix}
     1& 0 & 0\\
     0 & 1 & 0\\
     0 & 0 & 1
   \end{pmatrix}$ and $N_{13}=\begin{pmatrix}
     1 & k & 0\\
     \frac{1}{k} & 0 & 0\\
     0 & 0 & 1
   \end{pmatrix}$, where $k \in \mathbb{C}^*$.   
\item[(14)] $M_{14}=\begin{pmatrix}
     1& 0 & 0\\
     0 & 1 & 0\\
     0 & 0 & 1
   \end{pmatrix}$ and $N_{14}=\begin{pmatrix}
     1 & 0 & 0\\
     0 & 1 & 0\\
     0 & 0 & 1
   \end{pmatrix}$.            
\end{itemize}
\end{theorem}

\begin{proof}
The proof is similar to the proof of Theorem \ref{3locTn}.
\end{proof}

We now classify all homogeneous $3$-local representations of $WT_n$ for all $n\geq 4$.

\begin{theorem}\label{3locWTn}
Consider $n\geq 4$ and let $\gamma: WT_n\to \mathrm{GL}_{n+1}(\mathbb{C})$ be a homogeneous $3$-local representation of $WT_n$. Then, $\gamma$ is equivalent to one of the following five representations $\gamma_j$, $1\leq j\leq 5$, where $$\gamma_j(s_i)=\left( \begin{array}{c|@{}c|c@{}}
   \begin{matrix}
     I_{i-1} 
   \end{matrix} 
      & 0 & 0 \\
      \hline
    0 &\hspace{0.2cm} \begin{matrix}
   		M_j
   		\end{matrix}  & 0  \\
\hline
0 & 0 & I_{n-i-1}
\end{array} \right)
\text{ and  }\ \gamma_j(\rho_i)=\left( \begin{array}{c|@{}c|c@{}}
   \begin{matrix}
     I_{i-1} 
   \end{matrix} 
      & 0 & 0 \\
      \hline
    0 &\hspace{0.2cm} \begin{matrix}
   		N_j
   		\end{matrix}  & 0  \\
\hline
0 & 0 & I_{n-i-1}
\end{array} \right)$$
for all $1\leq i \leq n-1$,  and the matrices $M_j$'s and $N_j$'s are given below.
\begin{itemize}
\item[(1)] $M_1=\begin{pmatrix}
     1& 0 & 0\\
     0 & 0 & \frac{1}{h}\\
     0 & h & 0
   \end{pmatrix}$ and $N_1=\begin{pmatrix}
     1& 0 & 0\\
     0 & 0 & p\\
     0 & \frac{1}{p} & 0
   \end{pmatrix}$, where $h,p\in \mathbb{C}^*$.
\item[(2)] $M_2=\begin{pmatrix}
     0& \frac{1}{d} & 0\\
     d & 0 & 0\\
     0 & 0 & 1
   \end{pmatrix}$ and $N_2=\begin{pmatrix}
     0& k & 0\\
     \frac{1}{k} & 0 & 0\\
     0 & 0 & 1
   \end{pmatrix}$, where $d,k\in \mathbb{C}^*$.
\item[(3)] $M_3=\begin{pmatrix}
     1& 0 & 0\\
     \frac{1}{p} & -1 & p\\
     0 & 0 & 1
   \end{pmatrix}$ and $N_3=\begin{pmatrix}
     1& 0 & 0\\
     \frac{1}{p} & -1 & p\\
     0 & 0 & 1
   \end{pmatrix}$, where $p\in \mathbb{C}^*$.
\item[(4)] $M_4=\begin{pmatrix}
     1& k & 0\\
     0 & -1 & 0\\
     0 & \frac{1}{k} & 1
   \end{pmatrix}$ and $N_4=\begin{pmatrix}
     1& k & 0\\
     0 & -1 & 0\\
     0 & \frac{1}{k} & 1
   \end{pmatrix}$, where $k\in \mathbb{C}^*$.
\item[(5)] $M_5=\begin{pmatrix}
     1& 0 & 0\\
     0 & 1 & 0\\
     0 & 0 & 1
   \end{pmatrix}$ and $N_5=\begin{pmatrix}
     1 & 0 & 0\\
     0 & 1 & 0\\
     0 & 0 & 1
   \end{pmatrix}$.
\end{itemize}
\end{theorem}

\begin{proof}
The proof is similar to the proof of Theorem \ref{3locTn}.
\end{proof}

\subsection{On the irreducibility of the $3$-local representations of $VT_n$ and $WT_n$}

In this subsection, we prove that every homogeneous $3$-local representation of $VT_n$ or $WT_n$ is reducible for all $n\geq 4$. We start by the following theorem regarding the group $VT_n$.
 
\begin{theorem}\label{irred12}
Consider $n\geq 4$ and let $\delta: VT_n\to \mathrm{GL}_{n+1}(\mathbb{C})$ be a homogeneous $3$-local representation of $VT_n$. Then, $\delta$ is reducible.
\end{theorem}

\begin{proof}
According to Theorem \ref{3locVTn}, we know that $\delta$ is equivalent to one of the representations $\delta_j, 1\leq j \leq 14$, and so we consider the following cases.
\begin{itemize}
\item[(1)] In the case $\delta$ is equivalent to $\delta_j, j=1,4,8,9,10,11,12,14,$ we can see that the vector $(1,0,\ldots,0)^T$ is invariant under $\delta_j(s_i)$ and $\delta_j(\rho_i)$ for all $1\leq i \leq n-1$. Thus, $\delta_j$ is reducible and so $\delta$ is reducible.
\item[(2)] In the case $\delta$ is equivalent to $\delta_j, j=2,5,6,13,$ we can see that the vector $(0,\ldots,0,1)^T$ is invariant under $\delta_j(s_i)$ and $\delta_j(\rho_i)$ for all $1\leq i \leq n-1$. Thus, $\delta_j$ is reducible and so $\delta$ is reducible.
\item[(3)] In the case $\delta$ is equivalent to $\delta_j, j=3,7,$ we can see that $(1,0,\ldots,0).\delta_j(s_i)=(1,0,\ldots,0)$ and $(1,0,\ldots,0).\delta_j(\rho_i)=(1,0,\ldots,0)$ for all $1\leq i \leq n-1$. Thus, $\delta_j$ is reducible and so $\delta$ is reducible.
\end{itemize}
\end{proof}

Now, the next theorem is regarding the irreducibility of homogeneous $3$-local representations of $WT_n$ for all $n\geq 4$.
 
\begin{theorem}\label{irred122}
Consider $n\geq 4$ and let $\gamma: WT_n\to \mathrm{GL}_{n+1}(\mathbb{C})$ be a homogeneous $3$-local representation of $VT_n$. Then, $\gamma$ is reducible.
\end{theorem}

\begin{proof}
The proof is similar to that of Theorem \ref{irred12}.
\end{proof}

\vspace*{0.1cm}

\subsection{On the faithfulness of the $3$-local representations of $VT_n$ and $WT_n$}

In this subsection, we study the faithfulness of all homogeneous $3$-local representations of $VT_n$ and $WT_n$ for all $n\geq 4$. We start by the following theorem regarding the group $VT_n$.

\begin{theorem}\label{fa3locVTn}
Consider $n\geq 4$ and let $\delta: VT_n\to \mathrm{GL}_{n+1}(\mathbb{C})$ be a homogeneous $3$-local representation of $VT_n$. By Theorem \ref{3locVTn}, $\delta$ is equivalent to one of the representations $\delta_j$, $1\leq j\leq 14$. The following hold true.
\begin{itemize}
\item[(1)] If $\delta$ is equivalent to $\delta_1$, then $\delta$ is unfaithful if $e=1$.
\item[(2)] If $\delta$ is equivalent to $\delta_2$, then $\delta$ is unfaithful if $e=1$.
\item[(3)] If $\delta$ is equivalent to $\delta_3$, then $\delta$ is unfaithful if $d=\frac{1}{p}$.
\item[(4)] If $\delta$ is equivalent to $\delta_4$, then $\delta$ is unfaithful if $h=\frac{1}{k}$.
\item[(5)] If $\delta$ is equivalent to $\delta_j, 5\leq j \leq 14$, then $\delta$ is unfaithful.
\end{itemize}
\end{theorem}

\begin{proof}
We consider each case separately.
\begin{itemize}
\item[(1)] In the case $\delta$ is equivalent to $\delta_1$ and $e=1$, we have $\delta_1((s_i\rho_{i+1})^4)=I_{n+1}$ for all $1\leq i \leq n-2$ with $(s_i\rho_{i+1})^4$ are nontrivial elements in $VT_n$. Hence, $\delta_1$ is unfaithful and so $\delta$ is unfaithful.
\item[(2)] In the case $\delta$ is equivalent to $\delta_2$ and $e=1$, we have $\delta_2((s_{i+1}\rho_{i})^4)=I_{n+1}$ for all $1\leq i \leq n-2$ with $(s_{i+1}\rho_{i})^4$ are nontrivial elements in $VT_n$. Hence, $\delta_2$ is unfaithful and so $\delta$ is unfaithful.
\item[(3)] In the case $\delta$ is equivalent to $\delta_3$ and $d=\frac{1}{p}$, we have $\delta_3(s_i)=\delta_3(\rho_i)$ for all $1\leq i \leq n-1$ with $s_i\neq \rho_i$. Hence, $\delta_3$ is unfaithful and so $\delta$ is unfaithful.
\item[(4)] In the case $\delta$ is equivalent to $\delta_4$ and $h=\frac{1}{k}$, we have $\delta_4(s_i)=\delta_4(\rho_i)$ for all $1\leq i \leq n-1$ with $s_i\neq \rho_i$. Hence, $\delta_4$ is unfaithful and so $\delta$ is unfaithful.
\item[(5)] In the case $\delta$ is equivalent to $\delta_j, 5\leq j \leq 14$, we consider the following subcases.
\begin{itemize}
\item[•] If $\delta$ is equivalent to $\delta_5$ then we have $\delta_5((s_i\rho_{i+1})^4)=I_{n+1}$ for all $1\leq i \leq n-2$ with $(s_i\rho_{i+1})^4$ are nontrivial elements in $VT_n$. Hence, $\delta_5$ is unfaithful and so $\delta$ is unfaithful.
\item[•] If $\delta$ is equivalent to $\delta_6$ then we have $\delta_6((s_{i+1}\rho_{i})^4)=I_{n+1}$ for all $1\leq i \leq n-2$ with $(s_{i+1}\rho_{i})^4$ are nontrivial elements in $VT_n$. Hence, $\delta_6$ is unfaithful and so $\delta$ is unfaithful.
\item[•] If $\delta$ is equivalent to $\delta_7$ then we have $\delta_7(s_i)=I_{n+1}$ for all $1\leq i \leq n-1$. Hence, $\delta_7$ is unfaithful and so $\delta$ is unfaithful.
\item[•] If $\delta$ is equivalent to $\delta_8$ then we have $\delta_8(s_i)=I_{n+1}$ for all $1\leq i \leq n-1$. Hence, $\delta_8$ is unfaithful and so $\delta$ is unfaithful.
\item[•] If $\delta$ is equivalent to $\delta_9$ then we have $\delta_9((s_is_{i+1})^2)=I_{n+1}$ for all $1\leq i \leq n-2$ with $(s_is_{i+1})^2$ are nontrivial elements in $VT_n$. Hence, $\delta_9$ is unfaithful and so $\delta$ is unfaithful.
\item[•] If $\delta$ is equivalent to $\delta_{10}$ then we have $\delta_{10}((s_is_{i+1})^2)=I_{n+1}$ for all $1\leq i \leq n-2$ with $(s_is_{i+1})^2$ are nontrivial elements in $VT_n$. Hence, $\delta_{10}$ is unfaithful and so $\delta$ is unfaithful.
\item[•] If $\delta$ is equivalent to $\delta_{11}$ then we have $\delta_{11}(s_i)=I_{n+1}$ for all $1\leq i \leq n-1$. Hence, $\delta_{11}$ is unfaithful and so $\delta$ is unfaithful.
\item[•] If $\delta$ is equivalent to $\delta_{12}$ then we have $\delta_{12}((s_is_{i+1})^2)=I_{n+1}$ for all $1\leq i \leq n-2$ with $(s_is_{i+1})^2$ are nontrivial elements in $VT_n$. Hence, $\delta_{12}$ is unfaithful and so $\delta$ is unfaithful.
\item[•] If $\delta$ is equivalent to $\delta_{13}$ then we have $\delta_{13}(s_i)=I_{n+1}$ for all $1\leq i \leq n-1$. Hence, $\delta_{13}$ is unfaithful and so $\delta$ is unfaithful.
\item[•] If $\delta$ is equivalent to $\delta_{14}$ then we have $\delta_{14}(s_i)=I_{n+1}$ for all $1\leq i \leq n-1$. Hence, $\delta_{14}$ is unfaithful and so $\delta$ is unfaithful.
\end{itemize}
\end{itemize}
\end{proof}

Now, we study the faithfulness of all homogeneous $3$-local representations of $WT_n$ for all $n\geq 4$.

\begin{theorem}\label{fa3locWTn}
Consider $n\geq 4$ and let $\gamma: WT_n\to \mathrm{GL}_{n+1}(\mathbb{C})$ be a homogeneous $3$-local representation of $WT_n$. Then, $\gamma$ is unfaithful.
\end{theorem}

\begin{proof}
By Theorem \ref{3locWTn}, $\gamma$ is equivalent to one of the representations $\gamma_j$, $1\leq j\leq 5$. We consider each case separately.
\begin{itemize}
\item[(1)] In the case $\gamma$ is equivalent to $\gamma_j, j=1,2$, we have $\gamma_j(s_is_{i+1}s_i)=I_{n+1}$ for all $1\leq i \leq n-2$ with $s_is_{i+1}s_i$ are nontrivial elements in $WT_n$. Hence, $\gamma_j$ is unfaithful and so $\gamma$ is unfaithful.
\item[(2)] In the case $\gamma$ is equivalent to $\gamma_j, 3\leq j \leq 5$, we have $\gamma_j(s_i)=\gamma_j(\rho_i)$ for all $1 \leq i \leq n-1$ with $s_i\neq \rho_i$. Hence, $\gamma_j$ is unfaithful and so $\gamma$ is unfaithful.
\end{itemize}
\end{proof}

\vspace*{-0.4cm}

\section{Declaration section}
\subsection{Funding} Not Applicable.
\subsection{Ethical approval} Not Applicable.
\subsection{Informed consent} Not Applicable.
\subsection{Author Contributions} Not Applicable.
\subsection{Data Availability} Not Applicable.
\subsection{Conflict of Interest} The author declares that there is no conflict of interest.
\subsection{Clinical Trial Number} Not Applicable.

\end{document}